\documentclass[10pt]{article}
\usepackage{amsmath}
\usepackage{amsthm}
\usepackage{amsfonts}
\usepackage{mathrsfs}
\usepackage{amsfonts}
\usepackage{amssymb,amsfonts,amsmath, psfrag,eepic,colordvi,epsfig}
\usepackage{multirow} 
\usepackage{array} 

\usepackage{longtable}

\usepackage{cite}

\topskip 
\numberwithin{equation}{section}
\newtheorem{theorem}{Theorem}[section]

\newtheorem{conjecture}[theorem]{Conjecture}

\newcommand{\Rmnum}[1]{\uppercase\expandafter{\romannumeral #1}}

\begin{document}
	\parskip 7pt
	
	\pagenumbering{arabic}
	\def\sof{\hfill\rule{2mm}{2mm}}
	\def\ls{\leq}
	\def\gs{\geq}
	\def\SS{\mathcal S}
	\def\qq{{\bold q}}
	\def\MM{\mathcal M}
	\def\TT{\mathcal T}
	\def\EE{\mathcal E}
	\def\lsp{\mbox{lisp}}
	\def\rsp{\mbox{rasp}}
	\def\pf{\noindent {\it Proof.} }
	\def\mp{\mbox{pyramid}}
	\def\mb{\mbox{block}}
	\def\mc{\mbox{cross}}
	\def\qed{\hfill \rule{4pt}{7pt}}
	\def\pf{\noindent {\it Proof.} }
	\textheight=22cm

	\begin{center}
		{\Large\bf  A proof of  Andrews-El Bachraoui's
		  conjecture on the parity of coefficients 
		   of a $q$-series  }
	\end{center}
	
	\begin{center}

		Eric H. Liu$^{1}$ and Ernest X.W. Xia$^{2}$
		\\[10pt]

		$^{1}$School of Statistics and
		Information,\\
		Shanghai University of International
		Business and Economics,\\
		Shanghai, 201620, P. R. China
		
		$^{2}$School of Mathematical Sciences, \\
		Suzhou University of Science and
		Technology, \\
		Suzhou,  215009, Jiangsu,
		P. R. China

		Email: 
		$^{1}$liuhai@suibe.edu.cn,
		$^{2}$ernestxwxia@163.com

	\end{center}

		\noindent {\bf Abstract.}
			Recently, Andrews
		and El Bachraoui studied a partition function  $s_1(n)$, 
		which counts   the number of two-color partitions into
		distinct parts of $n$ whose smallest part occurs in one prescribed color only, while every
		larger part may occur in either color or in both colors.
		They obtained  a complete description modulo 4 for $s_1(n)$. They also considered  a $q$-series 
		$T_{o}(q)$  which is the   odd   companion series   of  the generating function for 
		$s_1(n)$.  At the end of their paper, they presented  a conjecture
		on the parity of the coefficients of $T_o(q)$.
		In this paper, we confirm  this conjecture.
		Moreover,  we establish  an infinite family 
		of congruences modulo 8 for the coefficients 
		of $S_1(q)$ and prove that  the set of integers satisfying $s_1(n)\equiv 0\pmod 8$ has natural density one.



		\section{Introduction }
		\allowdisplaybreaks
		
	In recent years, 
	 	integer partitions in which each part may occur in two colors (red and blue)  have
		been studied extensively.  For references on two-color partitions
		closely related to the present work, 
		    see, for example \cite{Andrews-Bachraoui,Banerjee-1,Chern,Sun,Xia}.
		Very recently,  Andrews
		and El Bachraoui \cite{Andrew} considered 
		 a series  $S_1(q)$ which is the generating 
		  function of a partition function $s_1(n)$.
    The partition function 		 
  $s_1(n)$ denotes    the number of two-color partitions into
		distinct parts of $n$ whose smallest part occurs in one prescribed color only, while every
		larger part may occur in either color or in both colors. 
		 	Recently, Andrews
		 and El Bachraoui \cite{Andrew} proved that 
		  \begin{align}
		  	\sum_{n=0}^\infty s_1(n)q^n:=S_1(q)
		  	 =\sum_{n=0}^\infty
		  	  (-q^{n+1};q)_\infty^2 q^n,\label{0-1}
		  \end{align}
	Here and throughout, we assume that  $q$
	is a complex number with $|q|<1$ 
	and adopt the following standard
	$q$-series notation:
	\begin{align*}
		(a;q)_\infty:=&\prod_{k=0}^\infty
		(1-aq^k),\qquad
		(a;q)_n:=\frac{(a;q)_\infty}
		{(aq^n;q)_\infty}.
	\end{align*}

	The series $S_1(q)$  also appears in Andrews'
	 work on concave and
	convex compositions \cite{Andrews2013}. Let $v_d(n)$
	 denotes the number of strictly concave
	  compositions of $n$. Andrews \cite{Andrews2013} proved that
	  \[
	  \sum_{n=0}^\infty v_d(n)q^n=S_1(q).
	  \]
	  Therefore $s_1(n)=v_d(n)$.  
	   Andrews
and El Bachraoui \cite{Andrew}  established 
 the following complete description modulo 4 for $s_1(n)$.
 
 \begin{theorem}
 	For $n\geq 0$, 
 	 	 \begin{align*}  
s_1(n)= \left\{
 	\begin{aligned}    
 		(-1)^m \pmod 4
 		,\qquad \qquad  &  {\rm if  \ }  n=m(m+1)/2\ {\rm for \ some}
 		\ m\in \mathbb{N},
 		\\
 		0 \pmod 4
 		,\qquad \qquad  &  {\rm otherwise }.
 	\end{aligned} \right.
 \end{align*} 
  \end{theorem}

     Andrews
  and El Bachraoui  \cite{Andrew}   considered 
  the odd and even companion series 
  \begin{align*}
	\sum_{n=0}^\infty t_e(n) q^n:=T_e(q)=
\sum_{m=0}^\infty \frac{(-1;q)_{2m}}{(q^2;q^2)_{m }}q^{2m}
  \end{align*}
 and 
  \begin{align}\label{m-2}
	\sum_{n=0}^\infty t_o(n) q^n:=T_o(q)=
	\sum_{m=0}^\infty \frac{(-q;q)_{2m}}{(q;q^2)_{m+1}}q^{2m}.
	\end{align}
 Andrews
and El Bachraoui \cite{Andrew}  also 	 proved that 
	 \[
	 S_1(q)=\frac{(q^2;q^2)_\infty^2}{(q;q)_\infty}T_e(q)+2qT_o(q).
	 \]
	
At the end of their paper \cite{Andrew}, 	 Andrews
and El Bachraoui   presented the following conjecture
  on  the parity of the coefficients 
  of $S_1(q)$. 
		
		\begin{conjecture}\label{C-1}
If $8n +9$ has a prime divisor $p \equiv  5, 7 \pmod 8$ raised  to odd exponent,
 equivalently,  if $8n+9$ is not represented by  $x^2+2y^2$, 
 then
	\begin{align}\label{t-1}
 t_o(n) \equiv 0 \pmod 2. 
 \end{align}
		\end{conjecture}
		
		The first goal of this paper is to confirm 
		 Conjecture \ref{C-1}. 
		
\begin{theorem}\label{Th-1}
	Conjecture \ref{C-1} is true. 
\end{theorem}
{\bf 
Remark}. After the present manuscript was submitted to ArXiv, Banerjee informed us that he and Bringmann \cite{Banerjee} had also proved Conjecture \ref{C-1} independently and simultaneously, using substantially different methods.

We also prove an infinite families 
 of congruences modulo 8.

\begin{theorem}\label{Th-2}
	Let $p$ be a prime with $p\equiv 5,7 \pmod 8$.
	Then for $n,k\geq 0$ with $p\nmid n$,
	\begin{align}\label{t-2}
	s_1\left(p^{2k+1}n +\frac{p^{2k+2}-1}{8}\right) \equiv 0 \pmod 8.
	\end{align}
\end{theorem}

For example, setting $p=5$ and $k=0$ in \eqref{t-2}, we  deduce 
 that for $n\geq 0$, 
 \[
 s_1(25n+8) \equiv  s_1(25n+13)\equiv  s_1(25n+18)\equiv  s_1(25n+23)
 \equiv 0\pmod 8.
 \]

\begin{theorem}\label{Th-3}
	We have 
	\begin{align}\label{t-3}
	\lim\limits_{n \rightarrow \infty }\frac{\{m|s_1(m) \not\equiv 0 \pmod 8, 0\leq m\leq n\}}{n}=0. 
	\end{align}
	
\end{theorem}

\section{Proof of Theorem \ref{Th-1}}

\noindent{\it Proof of Theorem \ref{Th-1}}. 
 Euler's pentagonal number theorem states that
		\begin{align}\label{1-1}
			(q;q)_\infty&=\sum_{n=0}^\infty (-1)^n q^{n(3n+1)/2}(1-q^{2n+1})
			\nonumber\\
			&= \sum_{n{\rm \ is\ even},\atop
			n\geq 0}   q^{n(3n+1)/2}(1-q^{2n+1})
			-\sum_{n{\rm \ is\ odd},\atop
				n\geq 1}   q^{n(3n+1)/2}(1-q^{2n+1}).
		\end{align}
		Andrews and El Bachraoui \cite{Andrew} 
		 introduced  a series
		 \[
		 	P(q)= \sum_{n=0}^\infty  q^{n(3n+1)/2}(1-q^{2n+1}).
		 \]
		 It is easy to see that 
		\begin{align}\label{1-2}
			P(q)
			&= \sum_{n{\rm \ is\ even},\atop
				n\geq 0}   q^{n(3n+1)/2}(1-q^{2n+1})
			+ \sum_{n{\rm \ is\ odd},\atop
				n\geq 1}   q^{n(3n+1)/2}(1-q^{2n+1}).
		\end{align}
		Combining \eqref{1-1} and \eqref{1-2} yields 
		\begin{align}\label{1-3}
			P(q)&=2\sum_{n{\rm \ is\ even},\atop
				n\geq 0}   q^{n(3n+1)/2}(1-q^{2n+1})-(q;q)_\infty\nonumber\\
				&=2\sum_{n=0}^\infty 
				 q^{n(6n+1)}(1-q^{4n+1})-(q;q)_\infty. 
		\end{align}
			Andrews and El Bachraoui \cite{Andrew}  showed 
			 that 
			\begin{align}\label{1-4}
			 2qT_o(q)=\frac{(q^2;q^2)_\infty^2
			 }{(q;q)_\infty^2 }P(q)-\sum_{n=0}^\infty (-1)^n q^{n(n+1)/2}
			 \end{align}
		 Substituting \eqref{1-4} into \eqref{1-3} yields 
		 \begin{align}\label{1-5}
		 	 2qT_o(q)
		 	&=\frac{(q^2;q^2)_\infty^2
		 	 }{(q;q)_\infty^2 }\left(2\sum_{n=0}^\infty 
		 	 q^{n(6n+1)}(1-q^{4n+1})-(q;q)_\infty\right)
		 	 -\sum_{n=0}^\infty (-1)^n q^{n(n+1)/2}
		\nonumber\\
		 	&=2\frac{(q^2;q^2)_\infty^2
		 	}{(q;q)_\infty^2 } \sum_{n=0}^\infty 
		 	q^{n(6n+1)}(1-q^{4n+1}) -\frac{(q^2;q^2)_\infty^2}{(q;q)_\infty}
		 	-\sum_{n=0}^\infty (-1)^n q^{n(n+1)/2}. 
		 \end{align}
		 Note that 
		 \begin{align}\label{1-6}
		 	\frac{(q^2;q^2)_\infty^2}{(q;q)_\infty}
		 +\sum_{n=0}^\infty (-1)^n q^{n(n+1)/2}
		 =	 \sum_{n=0}^\infty (1+(-1)^n) q^{n(n+1)/2}= 2\sum_{n=0}^\infty q^{n(2n+1)}.
		 \end{align}
		 It follows from \eqref{1-5} and \eqref{1-6} that 
		 	 \begin{align}\label{1-7}
		 	 qT_o(q)&= \frac{(q^2;q^2)_\infty^2
		 	}{(q;q)_\infty^2 } \sum_{n=0}^\infty 
		 	q^{n(6n+1)}(1-q^{4n+1}) -\sum_{n=0}^\infty q^{n(2n+1)}\nonumber\\
		 	&\equiv  (q^2;q^2)_\infty \sum_{n=0}^\infty 
		 	q^{n(6n+1)}(1+q^{4n+1}) +\sum_{n=0}^\infty q^{n(2n+1)} \pmod 2. 
		 \end{align}
		 Here we use  the fact that 
		 \[
		 (q;q)_\infty^2\equiv (q^2;q^2)_\infty \pmod 2.\]
		 	 Note that 
		 \begin{align}\label{1-8}
		 (q^2;q^2)_\infty =\sum_{k=-\infty}^\infty (-1)^k q^{3k^2+k}
		 \equiv \sum_{k=-\infty}^\infty   q^{3k^2+k}  \pmod 2.
		 \end{align}
It follows from 
\eqref{1-7} and \eqref{1-8} that 
		 \begin{align}\label{v-2}
		 \sum_{n=0}^\infty t_o(n) q^{n+1}
		&\equiv \sum_{k=-\infty }^\infty q^{k(3k+1) }
		\sum_{m=0}^\infty q^{6m^2+m} \nonumber\\
		&\quad +
		\sum_{k=-\infty }^\infty q^{k(3k+1) }
		\sum_{m=0}^\infty q^{6m^2+5m+1}+\sum_{m=0}^\infty q^{m(2m+1)} \pmod 2.
		 \end{align}
	 Therefore, 
		\begin{align}\label{1-9}
		 t_o(n)&\equiv \sum_{k(3k+1)+6m^2+m=n+1,\atop
		 (k,m)\in \mathbb{Z}\times\mathbb{N}  }1+
		  \sum_{k(3k+1)+6m^2+5m+1=n+1,\atop
		 	(k,m)\in \mathbb{Z}\times\mathbb{N}  }1+  \chi(n ) \pmod 2 \nonumber\\
		 	&=  \sum_{2(6k+1)^2+(12m+1)^2=3(8n+9),\atop
		 		(k,m)\in \mathbb{Z}\times\mathbb{N}  } 1+
		 	\sum_{2(6k+1)^2+(12m+5)^2=3(8n+9),\atop
		 		(k,m)\in \mathbb{Z}\times\mathbb{N}  }1+  \chi(n)
		 \end{align}
		where  
		 	 \begin{align}  \label{v-3}
		 	\chi(n)= \left\{
		 	\begin{aligned}    
		 		1
		 		,\qquad \qquad  &  {\rm if  \ }  n+1=2m^2+m\ {\rm for \ some}
		 		\ m\in \mathbb{N},
		 		\\
		 		0
		 		,\qquad \qquad  &  {\rm otherwise }.
		 	\end{aligned} \right.
		 \end{align} 
		Observe that 
		 \begin{align*} 
		 	\chi(n)= \left\{
		 	\begin{aligned}    
		 		1
		 		,\qquad \qquad  &  {\rm if  \ }  8n+9=(4m+1)^2 \ {\rm for \ some}
		 		\ m\in \mathbb{N},
		 		\\
		 		0
		 		,\qquad \qquad  &  {\rm otherwise }.
		 	\end{aligned} \right.
		 \end{align*}

		 We   recall a classical result on the number of representations of a natural
		 number by the quadratic forms $x^2+2y^2$ with $x,y\in \mathbb{Z}$.
	 This result can be found in
		 many places, such as \cite{OEIS}. 
		    Let $n$ have the prime factorization
		     $n=2^{\gamma}p_1^{\alpha_1}p_2^{\alpha_2}\cdots
		      p_k^{\alpha_k}q_1^{\beta_1}q_2^{\beta_2}
		      \cdots q_{\ell}^{\beta_{\ell}}$
		       with the $p_i \equiv 1,3\pmod 8$
		        and  $q_j \equiv 5,7\pmod 8$. If 
		  $r_2(n)$ denotes  the number of
		 solutions to $x^2 + 2y^2 =n$ with $x,y\in \mathbb{Z}$,
		 then 
		 \begin{align*} 
		 	r_2(n)= \left\{
		 	\begin{aligned}    
		 		 0
		 		,\qquad \qquad  &  {\rm if\ any \ } \beta_j\ {\rm is \ odd},
		 		\\
		 	2(\alpha_1+1)(\alpha_2+1)\cdots (\alpha_k+1)
		 		,\qquad \qquad  &  {\rm otherwise }.
		 	\end{aligned} \right.
		 \end{align*} 
		 From the above formula, it follows  that 
		 if  $r_2(8n+9)=0$, then $r_2(3(8n+9))=0$.
		  Moreover, if $r_2(3(8n+9))=0$, then 
		  \begin{align}\label{1-10}
		  \sum_{2(6k+1)^2+(12m+1)^2=3(8n+9),\atop
		  	(k,m)\in \mathbb{Z}\times\mathbb{N}  } 1=0,\quad 
		  \sum_{2(6k+1)^2+(12m+5)^2=3(8n+9),\atop
		  	(k,m)\in \mathbb{Z}\times\mathbb{N}  }1=0, \quad  \chi(n)=0.
		  \end{align}
		 It follows from
		  \eqref{1-9} and \eqref{1-10} that  if $r_2( 8n+9)=0$, then
		  \[
		  t_o(n) \equiv 0 \pmod 2.
		  \]
		  This completes the proof of Theorem \ref{Th-1}.  \qed

  \section{Proofs of Theorems \ref{Th-2}
   and \ref{Th-3}}
   
  We first give a proof of Theorem \ref{Th-2}.
  
  \noindent{\it Proof of Theorem \ref{Th-2}.}
   			Andrews and El Bachraoui \cite{Andrew}
   			 proved that 
 \begin{align}
 	S_1(q)=\sum_{n=0}^\infty (-1)^n q^{n(n+1)/2}
 	+4qT_o(q). \label{3-1}
 \end{align}
Combining \eqref{0-1}, \eqref{v-2} and \eqref{3-1} yields 
    \begin{align*}
  \sum_{n=0}^\infty s_1(n)q^n&=S_1(q)=\sum_{n=0}^\infty (-1)^n q^{n(n+1)/2}
   	+4\sum_{k=-\infty }^\infty q^{k(3k+1) }
   	\sum_{m=0}^\infty q^{6m^2+m}\nonumber\\
   	&\qquad +
   	4\sum_{k=-\infty }q^{k(3k+1) }
   	\sum_{m=0}^\infty q^{6m^2+5m+1}+4\sum_{m=0}^\infty q^{m(2m+1)} \pmod 8.
   \end{align*}
Therefore, 
   \begin{align}\label{3-2}
   	s_1(n)&\equiv  4 \sum_{k(3k+1)+6m^2+m=n ,\atop
   		(k,m)\in \mathbb{Z}\times\mathbb{N}  }1+
   	4 \sum_{k(3k+1)+6m^2+5m+1=n ,\atop
   		(k,m)\in \mathbb{Z}\times\mathbb{N}  }1+ 4 \chi(n-1 ) +\mu(n) \pmod 8 , 
   	\end{align}
   	   where $\chi(n)$ is defined by \eqref{v-3}
   	    and 
   	\begin{align}  \label{3-3}
   		\mu(n)= \left\{
   		\begin{aligned}    
   			(-1)^m
   			,\qquad \qquad  &  {\rm if  \ }  n =m(m+1)/2\ {\rm for \ some}
   			\ m\in \mathbb{N},
   			\\
   			0
   			,\qquad \qquad  &  {\rm otherwise }.
   		\end{aligned} \right.
   	\end{align} 
   	We can rewrite \eqref{3-2} as 
   	\begin{align}\label{3-4}
     	s_1(n)&\equiv 4   u(n)+ 4  \chi(n-1)+\mu(n) \pmod 8,
   \end{align}
   where 
   \begin{align}\label{mm-1}
   u(n)=\sum_{2(6k+1)^2+(12m+1)^2=24n+3,\atop
   	(k,m)\in \mathbb{Z}\times\mathbb{N}  } 1+
   \sum_{2(6k+1)^2+(12m+5)^2=24n+3,\atop
   	(k,m)\in \mathbb{Z}\times\mathbb{N}  }1.
   \end{align}
From the above identity, it follows  that if $24n+3$
   is not of the form
   $x^2+2y^2$, then $u(n)=0$.
Obverse  that if $N$ is   of the form $x^2+
   2y^2$,
   then $\nu_p(N)$ is even since
   $p$ is a prime
   with  $p\equiv 5,7 \pmod 8$ 
    and $\left(
   \frac{-2 }{p}\right)=-1$. Here $\nu_p(N)$
   denotes the highest power of $p$
   dividing
   $N$ and $\left(
   \frac{ \cdot }{p}\right)$ denotes
   the Legendre symbol.
   It is easy to verify that if $p\nmid n$,
   then
   \[
   \nu_p\left(24\left(p^{2\alpha+
   	1}n+\frac{ p^{2\alpha +2}-1
   }{8}\right)+3\right)=\nu_p(24 p^{2\alpha+1}n +3p^{2\alpha+2
   })=2\alpha+1
   \]
   is odd.
   Therefore,
   $24\left(p^{2\alpha+1}n+\frac{  p^{2\alpha+2 }-1 }{8
   }\right)+3$
   is not of the form
   $x^2+2y^2$ and
   \begin{align}\label{3-5}
   	u\left(p^{2\alpha+
   		1}n+\frac{ p^{2\alpha+2 }-1  }{8}\right)=0.
   \end{align}
Moreover,  it follows from \eqref{v-3} and \eqref{3-3} that 
 for $n\geq 0$ with $p\nmid n$,
\begin{align}
	\chi\left(p^{2\alpha+
		1}n+\frac{ p^{2\alpha+2 }-1  }{8}\right)&=0,\label{3-5}\\
	 \mu	\left(p^{2\alpha+
			1}n+\frac{ p^{2\alpha+2 }-1  }{8}\right)&=0.\label{3-6}
\end{align}
Congruence \eqref{t-2} follows from \eqref{3-4}--\eqref{3-6}. 
 This completes the proof of Theorem \ref{Th-2}. 	 
 
 Now, we turn to prove Theorem \ref{Th-3}.
 
 \noindent{\it Proof of Theorem \ref{Th-3}.}
 Recall  that an integral power series $ \sum_{n\in
 	\mathbb{Z}}c(n)q^n$ is called lacunary  if
 \[
 \lim_{X\rightarrow \infty} \frac{\#\{n|c(n)\neq
 	0, \ n\leq
 	X\}}{X}=0.
 \]
 
 To prove Theorem \ref{Th-3},
 we require
 the  following classical
 result due to Landau \cite{Landau};
 see also \cite{Serre}.
 
 	Let $r(n)$ and $s(n)$ be
 	quadratic polynomials. Then
 	\[
 	\left(\sum_{n=-\infty
 	}^\infty q^{r(n)}\right)
 	\left(\sum_{n=-\infty}^\infty q^{s(n)}\right)
 	\]
 	is lacunary. From this result and \eqref{mm-1},
 	 we know 
  \begin{align}
 \lim_{n\rightarrow \infty} \frac{\#\{m|u(m)\neq
 	0, 0\leq \ m\leq
 	n\}}{n}=0.\label{mm-2}
 \end{align}
 From the definitions of $\chi(n)$ and $\mu(n)$,
  we also have 
      \begin{align} \label{mm-3}
  \lim_{n\rightarrow \infty} \frac{\#\{m|\chi (m)\neq
  	0, 0\leq \ m\leq
  	n\}}{n}=  \lim_{n\rightarrow \infty} \frac{\#\{m|\mu (m)\neq
  	0, 0\leq \ m\leq
  	n\}}{n}=0.
  \end{align}
	Identity \eqref{t-3}
	 follows from \eqref{3-4}, \eqref{mm-2}
	  and \eqref{mm-3}. This completes 
	   the proof of Theorem \ref{Th-3}. \qed

	   \section{Conclusions}
	   
	   As seen in the Introduction, two-colored integer partitions have attracted considerable attention in recent years. 
	   In this paper, we prove
	   a conjecture   given by Andrews 
	   and  El Bachraoui \cite{Andrew}
  on the parity of the coefficients 
   of the series $T_o(q)$.  We also prove an infinite families
    of congruences modulo 8 for the partition function  $s_1(n)$
     and  show that  $\frac{s_1(n)}{8}$ takes integer  values with  natural density 1 for $n\geq 0$.
	   One natural research direction is to extend
	     these congruences for $s_1(n)$ to moduli 16, 32 and higher powers of 2.
	   In addition, it would be
	   interesting
	   to determine the
	   arithmetic density of   the
	   set of
	   integers such that
	   $s_1(n)\equiv 0 \pmod
	   {2^k}$ for some  fixed positive
	   integers $k\geq 4$.

	   \section*{Statements and Declarations}
	   
	   \noindent{\bf Acknowledgments}
	    This work was supported by 
	   the National Natural Science Foundation of
	   China  (grant
	   12371334) and the Qinglan project. 
	   
	   \noindent{\bf Competing Interests.}
	   The authors declare that they have
	   no conflict of interest.

	   \noindent{\bf Data Availability Statements.} Data sharing not applicable to this
	   article as no datasets were generated or analyzed during the current
	   study.

	\end{document}